\documentclass[12pt]{article}

\usepackage{graphicx}
\usepackage{enumerate}
\usepackage{amssymb}
\usepackage{amsmath}
\usepackage{mathrsfs}
\usepackage{latexsym}
\usepackage{psfrag}
\usepackage{mathrsfs}
\usepackage{verbatim}
\usepackage{xcolor}
\usepackage{comment}
\usepackage{algorithm}
\usepackage[noend]{algpseudocode}

\usepackage{enumerate,lineno,setspace,float}

\usepackage[utf8]{inputenc}
\usepackage{amsfonts, amsmath, amssymb, amsthm}
\usepackage{graphicx}
\usepackage{tikz, pgffor}
\usepackage{hyperref}
\usepackage[numbers,sort&compress]{natbib}
\usepackage[capitalize]{cleveref}
\hypersetup{
    colorlinks=true,       
    linkcolor=blue,        
    citecolor=blue,        
}

\allowdisplaybreaks

\usepackage{enumerate,lineno,setspace,float}
\usepackage{fullpage}

\usepackage[numbers, sort&compress]{natbib}
\usepackage[capitalize]{cleveref}
\hypersetup{
    colorlinks=true,       
    linkcolor=blue,        
    citecolor=blue,        
}

\newtheorem{theorem}{Theorem}
\newtheorem{lemma}[theorem]{Lemma}
\newtheorem{proposition}[theorem]{Proposition}
\newtheorem{corollary}[theorem]{Corollary}

\newtheorem{question}{Question}
\newtheorem{conjecture}{Conjecture}

\newcommand{\Href}[1]{\hyperref[#1]{\Cref{#1}}}
\renewcommand{\href}[1]{\hyperref[#1]{\cref{#1}}}
\renewcommand{\eqref}[1]{\hyperref[#1]{(\ref{#1})}}

\def\emptyset{\varnothing}
\def\dots{\ldots}
\def\leq{\leqslant}
\def\geq{\geqslant}

\def\deg{{\rm deg}}
\def\LO{{\rm LO}} 
\def\LEO{{\rm LO(E)}} 
\def\LOE{{\rm LE(O)}} 
\def\LEE{{\rm LE(E)}} 
\def\LOO{{\rm LO(O)}} 
\def\RO{{\rm RO}} 
\def\RE{{\rm RE}}
\def\REO{{\rm RO(E)}}
\def\REE{{\rm RE(E)}}
\def\ROE{{\rm RE(O)}}
\def\ROO{{\rm RO(O)}}
\def\LE{{\rm LE}} 
\def\R{{\rm R}} 
\def\L{{\rm L}}

\def\SS{{\rm S}(u,v)}
\def\S{{\rm S}(u)}
\def\E{{\rm E}}

\usepackage{color}
\definecolor{VeryLightBlue}{rgb}{0.9,0.9,1}
\definecolor{LightBlue}{rgb}{0.8,0.8,1}
\definecolor{MidBlue}{rgb}{0.5,0.5,1}
\definecolor{DarkBlue}{rgb}{0,0,0.6}
\definecolor{Blue}{rgb}{0,0,1}
\definecolor{Gold}{rgb}{1,0.843,0}
\definecolor{LightGreen}{rgb}{0.88,1,0.88}
\definecolor{MidGreen}{rgb}{0.6,1,0.6}
\definecolor{DarkGreen}{rgb}{0,0.6,0}
\definecolor{VeryLightYellow}{rgb}{1,1,0.9}
\definecolor{LightYellow}{rgb}{1,1,0.6}
\definecolor{MidYellow}{rgb}{1,1,0.5}
\definecolor{DarkYellow}{rgb}{1,1,0.01}
\definecolor{DarkPurple}{rgb}{.6,0,1}
\definecolor{Red}{rgb}{1,0,0}
\definecolor{VeryLightRed}{rgb}{1,0.9,0.9}
\definecolor{LightRed}{rgb}{1,0.8,0.8}
\definecolor{MidRed}{rgb}{1,0.55,0.55}

\title{An Inductive Approach to Strongly Antimagic \\
Labelings of Graphs}
\author
{Daphne Der-Fen Liu and Vicente Lossada\\ 
Department of Mathematics\\
California State University Los Angeles\\
{\small dliu@calstatela.edu; vicente.lossada@gmail.com}
}

\begin{document}

\baselineskip=14pt

\maketitle

\begin{abstract}
    An antimagic labeling for a graph $G$ with $m$ edges is a bijection $f: E(G) \to \{1, 2, \dots, m\}$ so that $\phi_f(u) \neq \phi_f(v)$ holds for any pair of distinct vertices $u, v \in V(G)$,  where $\phi_f(x) = \sum_{x \in e} f(e)$. A strongly antimagic labeling is an antimagic labeling with an additional condition: For any $u, v \in V(G)$, if  $\deg(u) > \deg(v)$,  then $\phi_f(u) > \phi_f(v)$. A graph $G$ is strongly antimagic if it admits a strongly antimagic labeling. We  present inductive properties of strongly antimagic labelings of graphs. This approach leads to simplified proofs that spiders and double spiders are strongly antimagic, previously shown by Shang \cite{Shang} and Huang \cite{Huang}, and by Chang, Chin, Li and Pan \cite{strongly}, respectively. We fix a subtle error in \cite{strongly}.  Further, we prove certain level-wise regular trees, cycle spiders and cycle double spiders are all  strongly antimagic. 
\end{abstract}

\section{Introduction}

For a positive integer $n$, denote $[n]=\{1, 2, \dots, n\}$. An {\it antimagic labeling} of a graph $G$  with $m$ edges is a bijection $f: E(G) \to [m]$ such that for every pair of distinct  vertices $u, v \in V(G)$,  $\phi_f(u) \neq \phi_f(v)$, where $\phi_f(x)$ is the sum of $f(e)$ over all edges $e$ incident to $x$. That is,  $\phi_f(x) = \sum_{x \in e}  f(e)$. 
When $f$ is clear from context, we simply denote $\phi_f(x)$ by $\phi(x)$. If $G$ admits an antimagic labeling, then $G$ is an {\it antimagic graph}. 
Given a labeling $f$ and a vertex $v$, we shall refer to the quantity $\phi_{f}(v)$ as the $\phi$\emph{-value} (or vertex-sum)  of $v$. We shall say the labeling $f$ of $G$ satisfies the \emph{antimagic property} for $G$ if the $\phi$-values of the vertices are all pairwise distinct.

The notion of antimagic labeling was introduced by Hartsfield and Ringel  \cite{Ringel}, who proposed the following conjecture:

\begin{conjecture}
\label{conj}
{\rm \cite{Ringel}} 
Every connected graph except $P_2$ is antimagic. 
\end{conjecture}

\noindent
In recent years, numerous researchers have investigated \Href{conj}. As a result of these efforts, many families of graphs are known to be antimagic. 
Alon, Kaplan, Lev, Roditty, and Yuster \cite{dense} proved that dense  graphs are antimagic. More precisely, the authors showed  that graphs of order $n$ and minimum degree  $\delta(G) \geq c \log n$ for some constant $c$, or with maximum degree $\Delta(G) \geq n-2$, are  antimagic. Other  families of graphs known to be antimagic include  regular graphs \cite{regular2, regular, regular bipartite}, trees with at most one vertex of degree two \cite{partition, trees}, subdivisions of trees \cite{trees}, and caterpillars \cite{caterpillar1,  caterpillars, caterpillar all}.

In \cite{Huang}, Huang introduced the notion of strongly antimagic labeling. A {\it strongly antimagic labeling} of $G$ is an antimagic labeling $f$ such that $\phi_f(u) > \phi_f(v)$ holds  whenever $\deg(u) > \deg(v)$.  
Considering strongly antimagic labelings can be beneficial for their use in solving other antimagic labeling problems. For example, for an integer $k$, a  {\it $k$-shifted antimagic labeling} of $G$ is  a bijection $f: E(G) \to [k+1, k+m]$ so that $f$ satisfies the antimagic property for $G$, where $[a,b]$ denotes $\{a, a+1, \dots, b\}$. If such a labeling exists then $G$ is called  {\it $k$-shifted antimagic}. Chang, Chen, Li and Pan \cite{shifted} studied the values of $k$ for a given graph $G$ to be $k$-shifted antimagic. In particular, the authors proved that if $G$ is strongly antimagic, then $G$ is $k$-shifted antimagic for all $k \geq 0$.

Moreover, strongly antimagic labelings also possess interesting inductive properties. 
Huang \cite{Huang} showed that if a tree is strongly antimagic, we may attach a pendant edge to each leaf of the tree to create a new tree that inherits the trait of being strongly antimagic. 
This inductive property was used to prove that every spider (a tree with exactly one vertex of degree greater than 2) is strongly antimagic. Chang, Chin, Li and Pan  \cite{strongly} also utilize this property to show  that every double spider (a tree with exactly two vertices of degree greater  than 2) is strongly antimagic. The proofs for both families of graphs are extensive. 

In this article, we aim to further explore and apply the inductive properties of strongly antimagic labelings. 
In Section 2, we establish general  inductive properties for strongly antimagic labelings of graphs. Applying these  results, immediately one can show that some level-wise regular trees are strongly antimagic, and in Sections 3 and 4,  respectively, we give simpler and more direct proofs that spiders and double spiders are strongly antimagic \cite{Huang, strongly}. We also correct a  subtle error in \cite{strongly}. 
Moreover, we prove in \Href{cycle double spider} that cycle double spiders (obtained by replacing every path, excluding the path connecting the two  center vertices, by a cycle) are strongly antimagic. 



In  Section 5, we discuss possible extensions of the inductive approach of strongly antimagic labeling to other graphs and raise some questions for future study.



\section{Inductive Properties}  


For a graph $G$, let $\delta(G)$ and $\Delta(G)$  denote the minimum and maximum degrees of $G$, respectively.  Let $V_i(G)$ be the set of vertices of degree $i$ in $G$. That is, $V_i(G) = \{v \in V(G): \deg(v) = i \}$. When $G$ is clear from context, we simply write $\Delta$ for $\Delta(G)$, and $V_i$ for $V_i(G)$.  A {\it pendant edge} is an edge that is incident to a leaf. 
For a vertex $v \in V(G)$, {\it adding a pendant  edge to $v$} means adding a new leaf $v_{0}$ adjacent to $v$. 
Similarly, {\it deleting a pendant edge from $v$}, say $e =vv_{0}$, means also deleting the leaf $v_{0}$. The first inductive property was due to Huang \cite{Huang} and Chang {\it et al.}  \cite{strongly}:  

\begin{theorem}
{\rm \cite{Huang, strongly}} 
\label{single}
Suppose $G$ is strongly antimagic with $V_i \neq \emptyset$ for some $i$. Then the super graph $G_i$ obtained by adding a pendant edge to each vertex in $V_i$ is strongly antimagic.
\end{theorem} 
  


We shall prove a more general result (\Href{big-one}), which is based on the following observation: any strongly antimagic labeling of $G$ induces a total ordering of $V(G)$. Namely, we define a relation $\preceq$ on $V(G)$ where $u \preceq v$ if and only if $\phi(v) \leq \phi(u)$. With this relation defined, if the set $\{\phi(v): v \in V(G)\}$ is ordered as 
$$
\phi_f(v_{1}) < \phi_f(v_{2}) \dots 
< \phi_f (v_i) < \phi_f(v_{i+1}) < 
\dots < \phi_f (v_{n}),  
$$
then the vertices are ordered as
\begin{equation}
\label{1}
v_{1} \prec v_{2} \dots 
\prec v_i \prec v_{i+1} \prec
\dots \prec v_{n},  
\end{equation}
where the vertices in $V_i$, if $V_i \neq \emptyset$, form a set of consecutive terms in the above ordering.

\begin{lemma}
\label{big-one} Let $G$ be a graph with $\delta(G) \geq 1$. 
Suppose $f$ is a strongly antimagic labeling for $G$ with the ordering in \href{1}. Assume $\deg(v_{j+1}) < \deg(v_{j+2})$ for some $j$, or $j+1=n$. Let  $G^*$ be the super graph of $G$ by adding an edge $e^{*} = uv_{j+1}$  where  
either $u = v_{j}$, if $v_jv_{j+1} \not\in E(G)$, or $u$ is a new vertex  $v^{*}$. Then $G^{*}$ is strongly antimagic. Moreover, there exists a strongly antimagic labeling for $G^*$, 
which preserves the ordering induced by $f$ on $V(G)$. 
\end{lemma}
\begin{proof}
 Define a labeling $f^*$ on $G^*$ by 
$$
f^*(e) = \left\{
\begin{array}{llll}
f(e)+1     &\mbox{if $e \in E(G)$}\\
1     &\mbox{if $e=e^*$.} 
\end{array}
\right.
$$
We claim the labeling $f^{*}$ is strongly antimagic for $G^{*}$.
 By definition,
$$
\phi_{f^*}(w) = \left\{
\begin{array}{llll}
\phi_f(w) + \deg_G(w)  &\mbox{if $w \neq v_{j+1}, u$, and $v_j$ (if $u=v_j$)}\\
\phi_f(w) + \deg_G(w)+1 &\mbox{if $w = u = v_{j}$ or $w=v_{j+1}$ }\\
1 &\mbox{if $w = u = v^{*}$}.
\end{array}
\right. 
$$
Note that as $\phi_{f}(v_{j-1}) < \phi_f(v_{j})$ and $\deg_G(v_{j-1}) \leq \deg_G(v_{j})$, we have $$
\phi_{f^*}(v_{j-1}) = \phi_{f}(v_{j-1}) + \deg_G(v_{j-1}) < \phi_{f}(v_{j}) + \deg_G(v_{j}) \leq \phi_{f^*}(v_{j}). 
$$ 
In addition, as $\phi_{f}(v_{j}) < \phi_{f}(v_{j+1})$ and $\deg_G(v_{j}) \leq \deg_G(v_{j+1})$, we get  
$$
\phi_{f^*}(v_{j}) \leq \phi_f(v_{j}) + \deg_G(v_{j})+1 < \phi_{f}(v_{j+1}) + \deg_{G}(v_{j+1}) + 1 = \phi_{f^*}(v_{j+1}). 
$$
If $j+1 = n$, there is nothing further to check. So assume $j+1 < n$. In this case, we have $\deg_{G}(v_{j+1}) + 1 \leq \deg_{G}(v_{j+2})$. Combining this with $\phi_{f}(v_{j+1}) < \phi_{f}(v_{j+2})$,  
we arrive at 
$$
\phi_{f^{*}}(v_{j+1}) = \phi_f(v_{j+1}) + \deg_G(v_{j+1})+1 < \phi_{f}(v_{j+2}) + \deg_{G}(v_{j+2}) = \phi_{f^*}(v_{j+2}).
$$

Now we consider vertices $v_s, v_t \in V(G)$ for some $v_s, v_t$ in \href{1} such that $s < t \leq j-1$. Then  $\deg_{G}(v_s) = \deg_{G^{*}}(v_s)$, $\deg_{G}(v_t) = \deg_{G^{*}}(v_t)$, and $\deg_{G}(v_s) \leq \deg_{G}(v_t)$. 
We simply note
$$
\phi_{f^{*}}(v_s) = \phi_{f}(v_s) + \deg_{G}(v_s) < \phi_{f}(v_t) + \deg_{G}(v_t) \leq  \phi_{f^{*}}(v_t).
$$
Similarly, one can show that $\phi_{f^*}(v_s) < \phi_{f^*} (v_t)$
 if $j+1\leq s < t $. 


Finally, we just note that if $u = v^{*}$, then $\phi_{f^{*}}(v^{*}) = 1$; this is the minimal $\phi$-value as every other vertex has a $\phi$-value of at least 2 under the labeling $f^{*}$
and $v^{*}$ has minimum degree.
Thus, in any case, the existing ordering of the vertices is preserved, possibly with one additional new vertex 
which will be at the beginning of the new ordering. 
\end{proof}

Sequentially adding pendant edges to the leaves of 3-path $P_3$, by \Href{big-one} we obtain

\begin{corollary}\rm
\label{p}
Every path $P_n$, $n \geq 3$, is strongly antimagic. 
\end{corollary}

A {\it complete level-wise regular tree} of height $h$, denoted by $T^1_{t_0, t_1, \ldots, t_{h-1}}$ (or $T^2_{t_0, t_1, \ldots, t_{h-1}}$, respectively), is a tree rooted at a  single vertex $w$ (or two vertices $w$ and $w'$, respectively) at level-0 where each vertex at level-$i$ (shorter distance $i$ from $w$ or $w'$) has degree $t_i$.  The following result can be obtained by \Href{big-one} or \Href{single}: 

\begin{corollary}
\label{level-wise regular trees}
Every complete level-wise regular tree $T^1_{t_0, t_1, \ldots, t_{h-1}}$ or $T^2_{t_0, t_1, \ldots, t_{h-1}}$ is strongly antimagic, if  
$
t_h \leq t_{h-1} \leq \cdots \leq t_1 \leq t_0. 
$
\end{corollary}

The following two corollaries will be especially useful in proving that every cycle double spider is strongly antimagic (see  \Href{cycle double spider}). 

\begin{corollary}
\label{add a cycle}
Let $G$ be a graph without leaves.  Assume $f$ is a strongly antimagic labeling for $G$ with  the order \href{1}. Suppose $j$ satisfies either $\deg(v_j) \leq \deg(v_{j+1}) -2$ or $j = |V(G)|$. Then for any positive integer $k \geq 3$, the super graph $G^{*}$ of $G$ obtained by attaching a $k$-cycle to $v_j$ is strongly antimagic. 
\end{corollary}

\begin{proof}
Initially, we add two new vertices $u_{k-1}$ and $u_1$ and two pendant edges $v_ju_1$ and $v_ju_{k-1}$. By our assumption that $\deg(v_j) \leq \deg(v_{j+1})-2$ or $j = |V(G)|$, \Href{big-one} implies that the  extended graph is strongly antimagic, where $u_1$ and $u_{k-1}$ are the only leaves. 

If $k=3$, by \Href{big-one}, the super graph $G^{*}$ formed by adding the edge $u_1u_2$ is strongly antimagic. If $k \geq 5$ and $k$ is  odd, we add two new vertices $u_2$ and $u_{k-2}$ along with pendant edges $u_1u_2$ and $u_{k-1}u_{k-2}$. Continue the same process of adding two new edges $u_iu_{i+1}$ and $u_{k-i}u_{k-i-1}$ along with new vertices $u_{i+1}$ and $u_{k-i-1}$ for $1 \leq i \leq \frac{k-3}{2}$. Finally we add the edge $u_{\frac{k-1}{2}} u_{\frac{k+1}{2}}$ to form a $k$-cycle incident to $v_j$.  By \Href{big-one} the final super graph of $G$ is strongly antimagic. 

Assume $k$ is even. Let $k=2k'$. We use the same process of adding edges $u_iu_{i+1}$ and $u_{k-i}u_{k-i-1}$ along with new vertices for $1 \leq i \leq k'-2$. Then we add only one edge $u_{k'-1}u_{k'}$ and one new vertex $u_{k'}$. 
Finally, we add the edge $u_{k'}u_{k'+1}$ to form a $k$-cycle incident to $v_j$. By \Href{big-one} the final super graph of $G$ is strongly antimagic. 
\end{proof}

\begin{corollary}
\label{add a path}
Let $G$ be a graph with exactly two  leaves $u$ and $v$.  Assume $f$ is a strong antimagic labeling for $G$. Then for any positive integer $k \geq 1$, the super graph $G^{*}$ of $G$ obtained by adding a new $uv$-path of length $k$ is strongly antimagic.
\end{corollary}

\section{Spiders and Cycle Spiders Are Strongly Antimagic}


Recall that a spider is a tree with exactly one vertex of degree greater than 2. Denote a spider by $\S$ with $\deg(u) \geq 3$ for some vertex $u$ called the {\it center} of $\S$. We view $\S$ as a tree that combines  $\deg(u)$  paths together by identifying one end vertex of each path to $u$. We call each of these paths a {\it leg} of $\S$. A star is a spider where each leg has length one. 

 Huang  \cite{Huang} and Shang \cite{Shang} independantly proved that every spider is strongly antimagic.  Shang's proof  gave an antimagic labeling of a spider. Huang's proof used  \Href{single}. Applying the inductive properties in the previous section, we present a more direct and simpler proof. 

\begin{lemma}\rm
\label{path}
For every degree-2 vertex $u$ on the  path $P_n$, $n \geq 3$, there exists a strongly antimagic labeling $f$ for $P_n$ with $\phi(u) = \max\{\phi(w): \deg(w)=2\}$. 
\end{lemma}

\begin{proof}
The case for $n=3$ is trivial. Assume $n \geq 4$.  Denote $V(P_n) = \{v_1, v_2, \dots, v_n\}$ and $E(P_n) = \{e_i: e_i = v_{i}v_{i+1}, 1 \leq i \leq n-1\}$. 

We first prove the case that  $u=v_{n-1}$ 
Define a function $f$ on $E(P_n)$ by
$$
f(e_i) = \left\{
\begin{array}{lll}
 \lfloor\frac{i+1}{2} \rfloor    &\mbox{if $n \not\equiv i$ (mod 2)}  \\ 
 \\
\frac{n+i}{2}     &\mbox{if $n \equiv i$ (mod 2)}. 
\end{array}
\right.
$$
It is easy to verify that $f$ is a strongly antimagic labeling for $P_n$ with $\phi(v_{n-1}) = \max\{\phi(w): \deg(w)=2\}$. 

It  remains to show that the statement holds when $u = v_{k}$ for some $\frac{n}{2} < k \leq n-2$; by symmetry the remaining cases   follow. 
According to the above discussion, there is a strongly antimagic labeling $f$ for the sub-path of $P_n$:  
$$
P'=\{v_{n-k}, v_{n-k+1}, \dots, v_{k}, v_{k+1}\}, 
$$
where $\phi_f(v_{k})=\max\{\phi_f(v): v \in V(P')\}$.  By repeatedly using \Href{big-one}, one can extend $f$ to a strongly antimagic labeling $g$ for $P_n$ so that  $\phi_{g}(u)$ remains the largest among all vertices in $P_n$. 
The proof is complete. 
\end{proof}

\begin{theorem} {\rm \cite{Shang, Huang}} 
\label{spider}
Spiders are strongly antimagic. 
\end{theorem} 

\begin{proof}
Let $\S$ be a spider with $\deg(u) \geq 3$.  Repeatedly in each round we delete one leaf from each leg of $\S$, until some leg becomes a pendant edge $e'$. Denote this reduced spider by $\S'$. By  \Href{single}, it is enough to show that $\S'$ is strongly antimagic. If $\S'$ has more than three legs, then by  
\Href{big-one}, it is enough to prove that $\S'-e'$ is strongly antimagic. Note, $\S'-e'$ has one leg less than $\S$ does.  
Repeat the entire process again to $\S' - e'$, until we obtain a spider $\S^*$ where $\deg_{\S^*} (u)= 3$ and $u$ is incident to a pendant edge $e^*$. Deleting $e^*$ from $\S^*$ we get a path.  
By \Href{big-one} and \Href{path}, $\S$ is strongly antimagic. 
\end{proof}

A {\it cycle spider} is established by replacing each leg of a spider by a cycle. 

\begin{corollary}
\label{cycle spider}
Every cycle spider is strongly antimagic.  
\end{corollary}

\begin{proof}
It is known \cite{regular} that regular graphs are antimagic. Hence they are strongly antimagic. Now start with a strongly antimagic labeling of a cycle $C_n$, where the vertex, say $u$, has the largest $\phi$-value.  We sequentially attach cycles to $u$ to obtain the given cycle spider, which by \Href{add a cycle} is strongly antimagic.   
\end{proof}


\section{Double Spiders and Cycle Double Spiders Are Strongly Antimagic}


Recall that a double spider is a tree with exactly two vertices of degree greater than two. Denote a double spider by $\SS$ where $\deg(u) \geq \deg(v) \geq 3$. We view  $\SS$ as a tree composed of three parts: middle, left, and right. Draw $\SS$ on the plane with $u$ on the left and $v$ on the right. The path from $u$ to $v$ is the {\it middle path}. A path from a leaf to $u$ that is also edge-disjoint from the middle path is said to be a {\it left leg} of the double spider. By replacing $u$ with $v$ in the definition of left leg, a {\it right leg} is defined similarly.

Prior to applying the labeling scheme, we apply a reduction algorithm to the double spider, repeatedly removing leaves and pendant edges until we are left with a double spider that fits into one of two maximally reduced cases. After we apply the labeling scheme to the reduced double spider, we obtain the strongly antimagic labeling of the original double spider by repeated applications of \Href{big-one}.

The reduction algorithm \Href{euclid} is defined below; its execution mainly consists of invoking two procedures: \textsc{DeleteLeaves()} and \textsc{DeletePendantEdge()}. The reduction algorithm generalizes the reduction previously applied to spiders in the proof of \href{spider}; the main difference is that the characterization of maximally reduced is more subtle for double spiders than for spiders. Let us briefly describe the auxiliary procedure \textsc{DeleteLeaves()}. The procedure shall only be used on double spiders $\SS$ such that neither $u$ nor $v$ is incident to a pendant edge. In such a double spider, every leaf is adjacent to a vertex of degree 2. The procedure \textsc{DeleteLeaves()} will remove every leaf and its incident pendant edge. Note that after using this procedure once on a double spider $\SS$, the number of leaves removed and the number of pendant edges removed are both $\deg(u)+\deg(v)-2$. The procedure \textsc{DeletePendantEdge()} takes as input a vertex $w$ that is adjacent to a leaf and removes the leaf adjacent to $w$ and its incident pendant edge. Further, by ``swap $u$ and $v$" in line 20 of \Href{euclid}, we simply mean to exchange the names. 

\begin{algorithm}[H]
\caption{Double Spider Reduction}
\label{euclid}
\begin{algorithmic}[5]

\Function{ReductionComplete}{$\SS$}
    \If {($\deg(v) = 3$) $\land$ \Call{HasLeaf}{$v$}}
        \If {$(\deg(u) < 5)$ $\lor$ $!$\Call{HasLeaf}{$u$}}
            
            \State \Return true
        \EndIf
    \ElsIf {$\deg(u) = \deg(v) + 1$ $\land$ \Call{HasLeaf}{$u$}}
        \If {$!$\Call{HasLeaf}{$v$}}
            
            \State \Return true
        \EndIf
    \EndIf
        
    \State \Return false

\EndFunction
\\
\Procedure{ReduceDoubleSpider}{$\SS$}
\Repeat
\While {$!$(\Call{HasLeaf}{$u$} $\lor$ \Call{HasLeaf}{$v$})}

\State \Call{DeleteLeaves}{$\SS$}
\EndWhile

\If {\Call{HasLeaf}{$v$}}
\If {$\deg(v) \geq 4$}

\State \Call{DeletePendantEdge}{$v$}
\EndIf
\EndIf

\If {\Call {HasLeaf}{$u$}}
    \If {$\deg(u) = \deg(v) = 4$}
        
        \State \Call{DeletePendantEdge}{$u$}
        
        \State Swap $u$ and $v$
    
    \ElsIf {$\deg(u) \geq \deg(v) + 2$}
    
        \State \Call{DeletePendantEdge}{$u$}
    \EndIf
\EndIf
\Until{\Call{ReductionComplete}{$\SS$}}
\EndProcedure
\end{algorithmic}
\end{algorithm}

The reduction algorithm  will repeatedly call \textsc{DeleteLeaves()} until $u$ or $v$ is incident to a pendant edge. Once $u$ or $v$ is incident to a pendant edge, we check the relevant conditions to see if we may apply the procedure \textsc{DeletePendantEdge()} at one or both of $u$ and $v$. Then we check if the double spider is maximally reduced by using the function \textsc{ReductionComplete()}. Note that if we did not actually remove a pendant edge from $u$ or $v$, \textsc{ReductionComplete()} will always return true, and we will terminate the algorithm. However, if we did remove at least one pendant edge, we face two possibilities: either the double spider became maximally reduced after the removal(s) and we terminate the algorithm, or further removals are possible and we again try to execute \textsc{DeleteLeaves()}.

After \href{euclid}, it suffices to consider the final reduced spiders, described in the following:

\begin{proposition} 
\label{2 types}
To prove that all double spiders are strongly antimagic it is enough to show that the following double spiders $\SS$, $\deg(u) \geq \deg(v) \geq 3$, are strongly antimagic:  
\begin{enumerate} [{\rm (a)}]
    \item $\deg(v)=3$, and $v$ is adjacent to at least one leaf. Further, if $\deg(u) \geq 5$, then $u$ is not adjacent to a leaf. 
\item $\deg(u)=\deg(v)+1$, $u$ is adjacent to at least one leaf,  but $v$ is not adjacent to any leaf. 
\end{enumerate}
\end{proposition}

To tackle  \href{2 types} (a), we start with the following result for spiders: 

\begin{lemma} 
\label{spider-ext}
Suppose $\S$ is a spider with $\deg(u) \geq 3$ and one of the following holds:

\begin{enumerate} [{\rm (i)}]
    \item $\deg(u) \geq 3$, and $\S$ has at least two legs of lengths longer than 1; 
\item  $\deg(u) = 4$, and $\S$ has only one leg of length longer than 1. 
\end{enumerate}
Then for any degree-2 vertex $v$ that is not adjacent to a leaf, there  exists a strongly antimagic labeling $f$ of $\S$ with $\phi_f(v) = \max\{\phi_f(w): \deg(w) = 2\}$.  
\end{lemma}

Combined with \Href{big-one}, \Href{spider-ext} implies
\begin{corollary} 
\label{coro}
A double spider $\SS$ is strongly antimagic if $\deg(u) \geq \deg(v)$ $=3$, $v$ is adjacent to exactly one leaf, and at least one of the following holds: 
\begin{enumerate}[{\rm (i)}]
    \item $u$ is adjacent to at least one internal vertex, excluding the middle path;
\item $\deg(u) = 4$, and $u$ is adjacent to only leaves, excluding the middle path. 
\end{enumerate}
\end{corollary}

Assume $\SS$ is a double spider with  $\deg(u)=4 = \deg(v)+1$, and all vertices adjacent to $u$ or $v$ are leaves, except the middle path. Let $\SS' = \SS-e'$ where $e'$ is a pendant edge incident to $u$. By the symmetric roles played by $u$ and $v$ in  $\SS'$ 
and by \Href{big-one}, $\SS$ is strongly antimagic if $\SS'$ is.
Thus, to show that $\SS$ is strongly antimagic when $\deg(u) \geq 4 = \deg(v) + 1$ and all right legs are pendant edges, we may focus on the case when $u$ is adjacent to at least one internal vertex in addition to its neighbor on the middle path.

Together  with \Href{coro}, to prove \Href{2 types} (a), it  remains to show:


\begin{lemma} 
\label{a3}
A double spider $\SS$ 
with $\deg(u) \geq \deg(v)=3$ is strongly antimagic if any of the following holds:
\begin{enumerate}[{\rm (i)}]
    \item $\deg(u) \geq 4$, $v$ is adjacent to two leaves, and $u$ is adjacent to at least one internal vertex, besides the middle path.  
\item $\deg(u) = 3$, $v$ is adjacent to at least one leaf, and $u$ is adjacent to two leaves. 
\end{enumerate}
\end{lemma}

 \noindent
In the rest of the section, we present the proofs of  \Href{spider-ext}, \Href{a3}, and \Href{2 types} (b). 

\medskip
\noindent
{\it Proof of \Href{spider-ext})} 
Denote $m \geq 6$ the number of edges in $\S$. 
Let $P$ be the leg of $\S$ that includes $v$. 
We represent $\S$ by three parts: {\it left}, {\it middle}, and {\it right}. All the legs (paths) of $\S$ except $P$ are the {\it left} paths, the path from $u$ to $v$ on $P$  is the {\it middle} path, and the path from $v$ to the leaf on $P$ is the {\it right} path, denoted by $R$.

A path is called {\it even} or ${\it odd}$ depending on the parity of its length. Suppose among the left paths, there are $y$ pendant edges, $c$ odd paths of length at least 3, and $d$ even paths.  Denote the set of $y$ pendant edges by $Y$, the $c$ left odd paths (and the $d$ even-paths, respectively) by LO$_1$, LO$_2$, $\dots$, LO$_c$ (LE$_1$, LE$_2$, $\dots$, LE$_d$, respectively). Note, since $v$ is not adjacent to a leaf, the length of $\R$ is at least 2. We also write $\R$ as $\RO_1$ or $\RE_1$ when the length of $\R$ is odd or even,  respectively.

To simplify the notations we use $e_i$ to  represent the edges respectively on each path,  without indicating which path they reside on (since that will be clear from context). A left path of length $l \geq 2$ is denoted by $(u_0, u_1, u_2, \dots, u_l = u)$ with edges $e_{i}=u_{i-1} u_i$, $1 \leq i \leq l$.  The middle path (of length $x$) is denoted by  
$$
X: (u=w_0, w_1, w_2 \dots, w_{x-1}, w_x=v)
$$ 
with edges $e_1 = uw_1$ and $e_i = w_{i-1}w_{i}$ for $2 \leq i \leq x$. The right path $\R$ ($\RO_1$ or $\RE_1$) is denoted by: $v=v_0, v_1, v_2, \dots, v_k$ with edges $e_i=v_{i-1}v_{i}$. 

Note that every left path ends at $u$, while the (only) right path starts at $v$. 
An edge $e_i$ in $\R \cup \LO \cup \LE$ is called {\it even}  (or {\it odd}, respectively) if its index $i$ is {\it even} (or {\it odd}, respectively). For a path $\LO_j$, the set of even (or odd, respectively) edges is denoted by $\LEO_j$ (or $\LOO_j$, respectively). Similarly, $\LOE_j$ and $\LEE_j$ denote the sets of odd and even edges in $\LE_j$,  respectively.  Also, the sets of even and odd edges in $\RE_1$ (or $\RO_1$, respectively) are denoted by $\REE_1$ and $\ROE_1$ (or $\REO_1$ and $\ROO_1$, respectively), respectively. For the middle path ${\rm X}$,  
we define two ordered subsets, $X_1$ and $X_2$, as follows.
$$
X_1= \left\{
\begin{array}{lll}
\emptyset &\mbox{if $x \leq 3$;} \\
(e_{x-2}, e_{x-4}, \dots, e_{2}) &\mbox{if $x$ is even, $x \geq 4$;}\\
(e_{3}, e_{5}, \dots, e_{x-2}) &\mbox{if $x$ is odd, $x \geq 5$.}
\end{array} 
\right. 
$$ 
$$
X_2= \left\{
\begin{array}{lll}
\emptyset &\mbox{if $x=1$;}\\
(e_{x-1}, e_{x-3}, \dots, e_{1}) &\mbox{if $x$ is even;}\\
(e_{2}, e_{4}, \dots, e_{x-1}) &\mbox{if $x$ is odd, $x \geq 3$.}
\end{array} 
\right. \hspace{1in}
$$
Then
$$
X \setminus (X_1 \cup X_2)
= \left\{
\begin{array}{lll} 
\{e_x\} &\mbox{if $x$ is even or $x=1$;} \\
\{e_1, e_x\}  &\mbox{if $x$ is odd, $x \geq 3$.}
\end{array}
\right. \hspace{1in}
$$

Next we define a labeling scheme, called Labeling A, which is defined by the following linear order of the edges. 
We label the edges with labels 1 through $m=|E(\S)|$, following the ordering so that the $j$th edge in the order receives label $j$. 
If  one set is empty, we proceed to the next. The order is organized in three levels, called Phases I, II, and III.

\medskip

\noindent
\fbox{{\bf Labeling A}}

\medskip

\noindent
$$
\begin{array}{lllll}
&&\ROO^1 \to \LOO^{-1} \to X_1 \to \REE \to \LOE \to Y    &&\mbox{(Phase I)} \\ 
&\to& \REO \to \LEO \to X_2 \to \ROE \to \LEE  &&\mbox{(Phase II)}\\
&\to&Z &&\mbox{(Phase III)}\\
\end{array}
$$

In the above, $\ROO^1$ means that we label all odd edges except the first odd edge in $\RO_1$ (if $\R$ is an odd path),  in the order $e_3, e_5, \ldots$, till the end of the last edge in $\RO_1$. Similarly, $\LOO^{-1}$ labels all odd edges in each $\LO_1, \LO_2, \ldots, \LO_c$,  one path after another (in the increasing order of the indices of the edges), skipping the {\it last} edge in $\LO_1$. For other sets in the ordering, we label edges without exclusion. For instance, $\LEO$ means we label even edges in the odd left paths $\LEO_1, \LEO_2, \ldots, \LEO_c$, one path after another, in this order; and for edges in $\LEO_i$ we label them according to the indices of the edges in $\LEO_i$,  in increasing order.  

After Phases I and II, the remaining unlabeled edges are denoted by  $Z = \{z_1, z_2, z_3, z_4\}$ where $z_1$ is the first edge ($e_1$) in $\RO_1$ (if exists), and $z_2$ is the last edge in $\LO_1$ (if exists). If $x=1$ or $x$ is even,  then $z_3=e_x \in X$, and $z_4$ does not exist. If $x$ is odd and $x \geq 3$, then $z_3=e_1 \in X$ and $z_4=e_x \in X$. Hence, $1 \leq |Z| \leq 4$. 
We label $Z$ in the order of $z_1, z_2, z_3, z_4$. If $z_i$ does not exist for some $i$, we move on to $z_{i+1}$, until all edges in $Z$ are labeled.  
See \Href{Label A} as an example.

\begin{figure}[h]
\centering
\small
\begin{tikzpicture}[scale = 0.6]
\draw (-6,8) -- (16,8); 
\draw (0,8) -- (0,11);
\draw (0,8) -- (2.8,10.8);
\draw[fill=black] (-6,8) circle (0.1);
\draw[fill=black] (-4,8) circle (0.1);
\draw[fill=black] (-2,8) circle (0.1);
\draw[fill=black] (2.8,10.8) circle (0.1);
\draw[fill=black] (0,8) circle (0.1); 
\draw[fill=black] (2,8) circle (0.1);
\draw[fill=black] (4,8) circle (0.1);
\draw[fill=black] (6,8) circle (0.1);
\draw[fill=black] (8,8) circle (0.1);
\draw[fill=black] (10,8) circle (0.1);
\draw[fill=black] (12,8) circle (0.1);
\draw[fill=black] (14,8) circle (0.1);
\draw[fill=black] (16,8) circle (0.1); 
  \draw[fill=black] (1.5,9.5) circle (0.1);
    \draw[fill=black] (0, 9.8) circle (0.1);
      \draw[fill=black] (0,11) circle (0.1);
	\node [left, red] at (8.5, 9) {{\Large $v$}};
\node [right] at (10.5, 8.3) {$3$};
			\node [right, red] at (3.8, 10) { $\phi(v)=25=\max\{\phi(w): \deg(w)=2\}$};
		\node [right] at (14.5, 8.3) {$4$};
		\node [right] at (2.5, 8.3) {$2$};
	\node [right] at (-5.8, 8.3) {$1$};
		\node [right] at (-0.6, 10.3) {$5$};
		\node [right] at (1.5, 10.3) {$6$};
	\node [right] at (0.2, 9.2) {$13$};
	\draw (0.8,9.2) circle (0.4);
		\node [right] at (-1.5, 8.3) {$14$}; 
\draw (-1,8.3) circle (0.4); 
\draw (-1,8.3) circle (0.5); 
		\node [right] at (-3.3, 8.3) {$7$}; 
\draw (-3,8.3) circle (0.4);
	\node [right] at (0.9, 8.3) {$9$};
	\draw (1.2,8.3) circle (0.4);
		\node [right] at (4.5, 8.3) {$8$};
	\draw (4.9,8.3) circle (0.4);
		\node [right] at (8.6, 8.3) {$10$};
	\draw (9.2,8.3) circle (0.4);
		\node [right] at (12.5, 8.3) {$11$};
	\draw (13,8.3) circle (0.4);
		\node [right] at (-0.9, 9.2) {$12$};
\draw (-0.3,9.2) circle (0.4);
	  		\node [right] at (6.5, 8.3) {$15$};
\draw (7,8.3) circle (0.4);
\draw (7,8.3) circle (0.5);
\end{tikzpicture}
\caption {\small{Labeling A applied to $\S$,  $\R=\RE_1$, $|\RE_1|=|X|=4$, $|Y|=0$, $|\LO|=1$, and $|\LE|=|Z|=2$. The circled  and double circled numbers are labeled in Phases II and III, respectively; the rest are labeled in Phase I.}}
\label{Label A}
\end{figure}
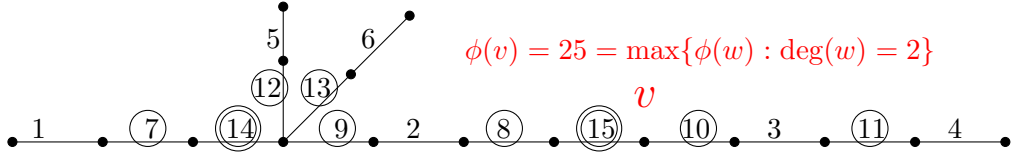

We make some simple observations on the general structure of the ordering defining the labeling scheme. First, each left leg or right leg of length at least 2 splits into a set of even edges and a set of odd edges; exactly one of these subsets contains one or more pendant edges. We assign one subset to Phase I and the other to Phase II, and we note that the choice is made so that the subset containing pendant edges always belongs to Phase I. In the same vein, the middle path (except one or two edges) splits into even edges and odd edges, and we choose the one whose parity matches the parity of $|X|$ for Phase I. Second, within a given phase, the left and right odd paths always precede the left and right even paths, and their labelings are separated by the labeling of a subset of the middle path. Third, for fixed path parity and edge parity, the edges on the right are always labeled before the edges on the left. Finally, $Y$ is the last set containing pendant edges to be labeled, and $Y$ is labeled before any path has been labeled completely. Altogether, the above observations, along with the appropriate indexing of individual paths, can be used as conditions to define an ordering nearly identical to the one given in the labeling scheme. The ordering of the labeling scheme differs only in that at least 1 and up to 4 edges have been promoted to the end of the ordering, so that they receive the last few labels.

A key idea in Labeling A is that the labels are assigned in an {\it alternating} and {\it parallel} way. 
In Phase I, we label about half of the edges in $G$. When two consecutive labels are  assigned on a path they are assigned to two edges of the same parity. That is, there is a {\it gap edge} between these two edges. We call this an {\it  alternating order}. In Phase II we fill up these gap edges, following the same order as in Phase I. This  alternating property is also kept for the  edges in $Z$ in Phase III.  
Due to the alternating and parallel order, it is clear that all vertices of degree 2 have distinct sums.  

Now we prove (i). By \Href{big-one}, it is enough to consider $|Y| \leq 1$; otherwise, since $u$ has at least two legs of length longer than 1, $|Y| \geq 2$ implies $\deg(u) \geq 4$. We could then remove $|Y|-1$ of these pendant edges via \Href{big-one} as $u$ will remain the vertex of maximal degree after the removals (any other vertex has degree at most 2).

\medskip
\noindent
{\bf Case 1.} $Y=\emptyset$.  Apply Labeling A to $\S$. Observe that all pendant edges are labeled in Phase I with distinct labels, and so  $\phi(w') < \phi(w)$ for all $\deg(w')=1$  and $\deg(w) \geq 2$. 
Because Labeling A is parallel and alternating, 
all degree-2 vertices have distinct sums. 

Next we claim $\phi(v) = \max\{\phi(w): \deg(w) = 2\}$. Assume $\deg(w)=2$ and $w  \neq v$. 
Then $w$ is either incident to an edge  labeled in Phases I and another in Phase II, or one labeled in Phase II and the other in  Phase III.  To prove $\phi(w) < \phi(v)$ it is enough to consider the latter case, which occurs when $w$ is incident to an edge of $Z$. The  other edge incident to $w$ belongs to $X_{2}$, $\LEO$, or $\REO$. For the first two cases, the edge of $\R$ incident to $v$ shall be labeled after the Phase II edge incident to $w$; since the other edge incident to $v$ has label $m$, the largest label, we conclude $\phi(v) > \phi(w)$. For the third case, just observe that $v$ will receive two Phase III labels, and one of the labels is $m$, the largest label.

Now we show that $\phi(u) > \phi(v)$. Since $|Y|=0$, by the labeling scheme, it is enough to consider the case that $\deg(u)=3$ and $\LE = \emptyset$. Otherwise, at least three edges incident to $u$ are  labeled in Phases II or III, from which it is easily seen that  $\phi(u) > \phi(v)$, as $\phi(v) \leq 2m- 1$; and if 
$|\LE|=|\LO|=1$ or $|\LE| \geq 2$, then $\phi(u) \geq 2m$.

If $x$ is even and $\R$ is even (that is, $\R=\RE_1$), then $\phi(u) \geq (m-1) + 3  +  m-2- \frac{|\RE_1|}{2} = 2m - \frac{|\RE_1|}{2}$ and $\phi(v) = m + (m-1 -\frac{\RE_1}{2})   < \phi(u)$.   Assume $x$ is odd, or $x$ is even with $\R=\RO_1$. 
In the former case, $\phi(u) \geq (m-1) + (m-2) + 3 =2m > 2m-1 \geq \phi(v)$.  For the latter case, $\phi(u) \geq (m-1)+(m-3) + 4=2m > \phi(v)$. 

\medskip

\noindent
{\bf Case 2.} $|Y|=1$. By  \Href{big-one}, it is  enough to consider that  $\deg(u)=3$ (if $\deg(u) \geq 4$, then in $G-e^*$, $u$ is the only vertex of degree 3).  That is, the left side contains exactly one path of length longer than one, denoted by $\L$ (which is $\LE_1$ or $\LO_1$).  Denote  $par(\L)$-$par(X)$-$par(\R)$ the parity of the left path $\L$, the middle path $X$, and the right path $\R$, respectively. For instance, even-odd-even indicates that the left, middle, and right paths are even, odd, and even, respectively.  

We apply Labeling A to $\S$, with one exception, for the even-even-odd case we swap the labeling order of the last edge to be labeled in $\RO_1$ (that is, $e_1$ in $\RO_1$)  with the last edge in $\LEE_1$ in the ordering.  

To prove the result,  similar to Case 1, it is enough to show that $\phi(u) > \phi(v) > \phi(w)$,  if $\deg(w) \leq 2$. One can easily verify that $\phi(v) >\phi(w)$ if $\deg(w) \leq 2$. We now claim $\phi(u) > \phi(v)$ by listing the eight possibilities in the table below. Note that if $x=1$ then $\phi(u) > \phi(v)$ as both $u$ and $v$ are adjacent to the edge labeled by $m$.  In the table, we denote $l$ and $r$ the floors of the half lengths of the left and right paths, respectively. 
$$
\small
\begin{array}{|c|c|c|}\hline
\mbox{$\L$-$X$-$\R$} &\phi(u) &\phi(v) \\ \hline 
\mbox{even-even-even} &2m +  (m-1)/2 - l - r-2 & 2m - l - r \\ \hline 
\mbox{even-even-odd} &2m+(m/2)-l-4 &2m-2\\ \hline
  \mbox{odd-even-even} &2m + (m/2) - r-4 & 2m-r-1\\ \hline 
   \mbox{odd-even-odd} &2m + (m-11)/2 \geq 2m-1 \ (\because m \geq 9) & 2m-2\\ \hline 
\mbox{even-odd-even, $x \geq 3$} &2m+(m/2) - 4 &2m-l-r-1\\ \hline
\mbox{even-odd-odd, $x \geq 3$} &2m+(m-11)/2 \geq 2m-1 \ (\because m \geq 9)  &2m-2\\ \hline
\mbox{odd-odd-even, $x \geq 3$} & 2m+(m-9)/2 \geq 2m-2 &2m-r-2 \leq 2m-3\\ \hline 
\mbox{odd-odd-odd, $x \geq 3$} &2m+(m/2) - 5 \geq 2m-2 &2m-3\\ \hline
\end{array}
$$

Next we prove (ii). Assume $\deg(u) = 4$ and $\S$ has only one leg longer than one (on which $v$ resides). That is, all left paths are pendant edges,  $|Y|=3$. 
Apply Labeling A to $\S$. 
Let $e^{\prime}$ be the edge of the right path that is incident to $v$. Let $y_{1}, y_{2}, y_{3}$ be the pendant edges incident to $u$.
Then $\phi(u) = f(e_{1}) + f(y_{1}) + f(y_{2}) + f(y_{3})$ and $\phi(v) = m + f(e^{\prime})$.
Let $x^{\prime} = \lfloor{\frac{|X|}{2}}\rfloor$ and let $r = \lfloor{\frac{|\R|}{2}}\rfloor$.
Note that if $x^{\prime}=0$, then $f(e_{1}) = m$, $f(e^{\prime}) \leq m-1$, and the three pendant edges have labels whose sum exceeds $m$. Thus, we will have $\phi(u) > \phi(v)$.

Now we may assume that $x^{\prime} \geq 1$ and $r \geq 1$. We have $f(y_{1}) + f(y_{2}) + f(y_{3}) = 3(x^{\prime} + r + 1)$.
We may re-express this sum as $x^{\prime} + r + m - i_{X} - i_{\R}$, where each of $i_{X}$ and $i_{\R}$ is 1 if its corresponding path is odd, 0 otherwise. 
Then we note $f(e^{\prime}) \leq f(e_{1}) + 1 - 2i_{X}$. So $\phi(v) \leq m + f(e_{1}) + 1 -  2i_{X}$. Subtracting our upper bound for $\phi(v)$  from $\phi(u)$ yields $\phi(u)-\phi(v) \geq x^{\prime} + r + i_{X} - i_{\R} - 1$. We note that if $x^{\prime}+r \geq 3$, we are done. Otherwise, $x^{\prime}+r = 2$, and $i_{X} - i_{\R} \geq 0$ holds, except the case of $|X|=2$ and $|\R|=3$, for which a labeling given in \Href{8} (a). 
Hence, in any case, we have $\phi(u) - \phi(v) \geq 1$. 

Thus, the proof for \Href{spider-ext} is complete. 
\hfill$\blacksquare$

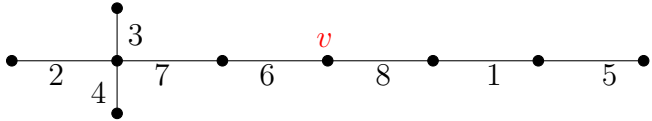
\begin{figure}[h]
\begin{tikzpicture}[scale = 0.7]
\draw (0,0) -- (12,0); 
\draw (2,1) -- (2,-1);
	\foreach \i in {0,2,4,6,8,10,12}
	{
		    \draw[fill=black] (0+\i, 0) circle (0.1);
	}
        \draw[fill=black] (2, 1) circle (0.1);
       \draw[fill=black] (2, -1) circle (0.1);
	\node [left,red] at (6.3, 0.4) {$v$};
	\node [right] at (11, -0.27) {$5$};
		\node [right] at (8.8, -0.27) {$1$};
		\node [right] at (6.7, -0.27) {$8$};
		\node [right] at (4.5, -0.27) {$6$};
	\node [right] at (2.5, -0.27) {$7$};
		\node [left] at (1.2, -0.27) {$2$};	
		\node [right] at (2, 0.5) {$3$};
		\node [left] at (2, -0.6) {$4$};

\end{tikzpicture}
\caption {\small{ $\phi(v)=14= \max\{\phi(w): \deg(w)=2\}$.}} 
\label{8}
\end{figure}


\medskip
\noindent
{\it Proof of \Href{a3})} Let $m$ be the number of edges in $\SS$. Assume (i) holds. 
Denote $Q$ a left leg of length at least two, and denote the two pendant edges incident to $v$ by $q, q'$. We apply Labeling A with two modifications. First, we label $q$ and $q^{\prime}$ at the end of Phase I, $f(q')=f(q)+1$, right after labeling $Y$.
Before describing the second modification, let us recall that Labeling A is defined by a linear ordering of the edges in which a single left odd edge incident to $u$ is removed from its natural place in the ordering with the other left odd edges and placed near the end of the ordering, succeeded only by one or two edges on the middle path.  Here, letting $a$ be the number of left odd paths with length greater than one, we instead move $min\{2, a\}$ left odd edges incident to $u$ to the end of the ordering so that only one or two middle path edges succeed them.
Here, letting $a$ be the number of left odd paths with length greater than one, we instead move $min\{2, a\}$ edges that are on left odd paths and incident to $u$ to the end of the ordering so that only one or two middle path edges succeed them.

To show that the modified labeling scheme A  produces a strongly antimagic labeling, it suffices to consider the case that $\deg(u)=4$ and prove that  
$\phi(u) > \phi(v) > \phi(w)$ holds for every degree-2 vertex $w$.  First we show $\phi(u) > \phi(v)$. 
We know that $\phi(v) = m + f(q) + f(q^{\prime})$. We make two simple observations for $u$: first, an edge incident to $u$ will receive the label $m-1$, and, second, there will exist another edge incident to $u$ whose label is at least $f(q^{\prime}) + 1$.
Thus, we only need that the other two labels on edges incident to $u$ sum to at least $f(q)+1$. If either of the left legs other than $Q$ has length greater than 1, the corresponding edge incident to $u$ will have a label strictly greater than $f(q) + 1$. This occurs regardless of the parity of the path length because we have modified the labeling scheme to save two odd edges incident to $u$ when possible.  Thus, we need only consider the case where both of the left legs other than $Q$ are pendant edges. In this case, their labels sum to $(f(q)-2) + (f(q)-1) = 2f(q) - 3$. We note that $2f(q) - 3 \geq f(q) + 1$ so long as $f(q) \geq 4$. But this always holds since at least 3 edges are always labeled before $q$.

Next we show that $\phi(v) > \phi(w)$ holds for any degree-2 vertex $w$. Since the maximal $\phi$-value of any vertex having degree-2 occurs at the vertex $w_{x-1}$ of the middle path $X$, we need only verify that $\phi(w_{x-1}) < \phi(v)$. Recall $\phi(v) = m + f(q) + f(q^{\prime})$. We also note that $\phi(w_{x-1}) = m + f(e_{x-1})$. Let $A$ be the set of edges labeled before $q$. Let $B$ be the set of edges labeled after $q^{\prime}$ and before $e_{x-1}$. Then $f(q) = |A| + 1$ and $f(e_{x-1}) = f(q^{\prime}) + |B| + 1$. Thus, to show that $\phi(v) > \phi(w_{x-1})$, it will suffice for us to show that $|A| > |B|$. We note that $A$ and $B$ are disjoint sets and neither contains $q$ or $q^{\prime}$. Let $P$, $Q$, and $S$ be the paths incident to $u$. We choose $Q$ and $P$ to be odd paths of length at least 3 if possible, giving priority to $Q$ if only one such path exists. We  compare the cardinalities of the sets $A$ and $B$ by examining their intersections with $Q$, $P$, $S$, and $X$. 
Note that $A = (Q \cap A) \dot{\sqcup} (P \cap A) \dot{\sqcup} (X \cap A) \dot{\sqcup} (S \cap A)$ and $B = (Q \cap B) \dot{\sqcup} (P \cap B) \dot{\sqcup} (X \cap B) \dot{\sqcup} (S \cap B)$. 
We make the following observations: 
$|Q \cap A| - |Q \cap B| \geq 0$, $|P \cap A| - |P \cap B| \geq 0$, $|S \cap A| - |S \cap B| \geq 1$, and $|X \cap A| - |X \cap B| \geq 0$. 
Combining the estimates yields $|A| - |B| > 0$.  Hence, $\phi(v) > \phi(w)$ for all degree-2 vertex $w$.

Next, assume (ii) holds.  First suppose $v$ is adjacent to exactly one leaf. Denote the two pendant edges incident to $u$ by $e', e''$ and the pendant edge incident to $v$ by $q$. Then $\SS - \{e', e'', q\}$ is a  path, denoted by $
P: u=w_0, w_1, w_2, \dots, w_k,
$ 
where $v=w_t$ for some $1 \leq t \leq k-2$. Denote the edges of $P$ by $e_i=w_{i-1}w_i$, $1 \leq i \leq k$. In each of the following cases, we define a labeling of $\SS$ by giving  an ordering of edges and assigning the labels $[1,m]$ according to the ordering, one by one:

If $k$ is even, define 
$$
\begin{array}{lll}
&\mbox{If $t > k/2$:} &(e_{2}, e_{4}, e_{6},  \dots, e_{k}, e', e'', q, e_{1}, e_{3}, e_5,  \dots, e_{k-1}). \\
&\mbox{If $t \leq k/2$:}  &(e_{k}, e_{k-2}, e_{k-4},  \dots,  e_{2}, q, e',  e'',   e_{k-1},  e_{k-3}, \dots, e_{1}). 
\end{array}
$$
We claim that each of the above gives a strongly antimagic labeling. If $t > k/2$, the labeling gives: 
$$
\phi(w) = \left\{ 
\begin{array}{lll} 
\frac{k}{2} + i + 4 \leq \frac{3k}{2}+3 &w=w_i \ \mbox{for some} \ i \in [k-1] \setminus \{t\},\\
k + t + 7 &w=v=w_t,\\ 
\frac{3k}{2}+7 &w=u.
\end{array}
\right.
$$
As $t > k/2$, one can easily see that $f$ is a strongly antimagic labeling for $\SS$.

The case for $t \leq k/2$ can be similarly proved, as the labeling gives:
$$
\phi(w) = \left\{ 
\begin{array}{lll} 
k/2 + 4 -i \leq (3k/2)+ 3 &w=w_i \ \mbox{for some} \ i \in [k-1] \setminus \{t\},\\
2k + 5 - t \geq (3k/2) + 5 &w=v=v_t,\\ 
2k+8 &w=u.
\end{array}
\right.
$$
Therefore, $f$ is a strongly antimagic labeling for $\SS$. 

Suppose $k$ is odd. If $t$ is odd, $1 \leq t \leq k-2$ and $t \neq (k-1)/2$,  define the ordering by: 
$$
(e_{k}, e_{k-2},  \dots,  e_{t+2}, e_{t-1}, e_{t-3}, \dots, e_{2}, q,  
e',  e'',  e_{k-1},  e_{k-3}, \dots, e_{t+1}, e_{t}, e_{t-2}, \dots, e_{1}).  
$$
For the special case when $t=(k-1)/2$ is odd,  we swap the edges $e'$ with $q$ in the above ordering. 

Now we show that the orderings above give strongly antimagic labelings.  If $t \neq (k-1)/2$, for vertices $w_i$, $i \in [k] \setminus \{t\}$, we have 
$ 
\phi(w)  \leq  (3k+5)/2,  
$
and they all have distinct sums. 
For $u, v$,  
$\phi(v) = \phi(v_t) = 2k+7 + (k-2t-1)/2
$ and 
$\phi(u) = 2k+7$.
Therefore, $f$ is a strongly antimagic labeling for $\SS$. For the case that $t=(k-1)/2$, due to the swapping of the edges $e'$ and $q$, we obtain $\phi(u)=2k+ 6$ and $\phi(v)=2k+8$. Hence $f$ is a strongly antimagic labeling for $\SS$. 

If $t$ is even and $6 \leq t \leq k-3$, we use the following ordering: 
$$ 
(e_{k}, e_{k-2},   \ldots,  e_{t+3}, e_{t-2},  \ldots, e_{2},  
q, e',  e'',  e_{k-1},  e_{k-3}, \ldots, e_{t+2}, e_{t-1}, \ldots, e_{1}, e_{t+1}, e_{t}).
$$
By calculation we obtain: 
For $w_i$, $i \in [k-1] \setminus \{t\}$,  
$\phi(w) \leq 2k - (t/2) + 5$, and they all have distinct sums.  
For $v=v_t$ and $u$, 
$
\phi(v)=2k+3 + (k+3)/2  
$ and 
$\phi(u) = 2k+3$. Thus $f$ is a strongly antimagic labeling for $\SS$.  

For $t=2,4$, we use the ordering:
$ 
(e_{k}, e_{k-2},   \dots, e_{1},  
e',  e'',  q, e_{k-1},  e_{k-3}, \dots, e_{2}).
$
Similarly, one can verify that $f$ is a strongly antimagic labeling.  
 
Finally, assume $v$ is incident to two pendant edges, $q, q'$. Denote the pendant edges incident to $u$ by $e', e''$, and the path connecting $u$ and $v$ by 
$ 
P: u=w_0, w_1, w_2, \dots, w_k=v;$ where edges of $P$ are  denoted by $e_i=w_{i-1}w_{i}$, $1 \leq i \leq k$. The labeling is by the following order: 
\begin{equation}
e_2, \ e_4, \ \dots, \ e_{2 \lfloor \frac{k}{2} \rfloor}, \ e', \ e'', \ q, \ q', \ e_1, \ e_3, \ \dots, \ e_{2 \lfloor \frac{k}{2}  \rfloor +1}, 
\label{double leaves}
\end{equation}
where the last term  
exists only when $k$ is odd. It is straightforward to check the above ordering provides a strongly antimagic labeling for $\SS$. Thus, the proof for \Href{a3} is complete. 
\hfill$\blacksquare$



\medskip

\noindent
{\it Proof of \Href{2 types} (b).}
Let $\SS$ be a double spider with $\deg(u)= \deg(v)+1 \geq 4$, where $u$ is adjacent to at least one leaf, but $v$ is not adjacent to any leaf. In the following we find a strongly antimagic labeling for $\SS$ by considering two cases:  

\noindent
{\bf Case 1.} Assume $v$ is incident to only odd paths of length at least 3. Apply Labeling A to $\SS$. Note that all edges incident to $v$ from the right paths, except $e_1$ in $\RO_1$, are labeled in Phase I,  before any edge incident to $u$ in the left side is labeled. 
See \Href{big A} for an example.

\begin{figure}[h]
\centering
\small
\begin{tikzpicture}[scale = 0.6]
\draw (-2,0) -- (20,0); 
\draw (0,0) -- (0,2);
\draw (0,0) -- (2,2);
\draw (10,0) -- (16,2);
\draw[fill=black] (-2,0) circle (0.15);
\draw[fill=black] (0,0) circle (0.15);
\draw[fill=black] (2,0) circle (0.15);
\draw[fill=black] (4,0) circle (0.15);
\draw[fill=black] (6,0) circle (0.15); 
\draw[fill=black] (8,0) circle (0.15);
\draw[fill=black] (10,0) circle (0.15);
\draw[fill=black] (12,0) circle (0.15);
\draw[fill=black] (14,0) circle (0.15);
\draw[fill=black] (16,0) circle (0.15);
\draw[fill=black] (18,0) circle (0.15);
\draw[fill=black] (20,0) circle (0.15);
\draw[fill=black] (0,2) circle (0.15); 
  \draw[fill=black] (2,2) circle (0.15);
   \draw[fill=black] (16, 2) circle (0.15);
     \draw[fill=black] (14,1.4) circle (0.15);
  \draw[fill=black] (12,0.7) circle (0.15);

\node [left] at (-1, 0.3) {$8$};
\node [right] at (-0.8, 1) {$7$};
		\node [right] at (0.4, 1.2) {$6$};
		\node [left] at (1.3, -0.5) {$15$};
	\node [right] at (8.5, -0.5) {$16$};	
	\node [right] at (2.7, 0.5) {$12$};
	\draw (3.3,0.5) circle (0.4); 
\node [right] at (6.7, 0.5) {$13$};
	\draw (7.3,0.5) circle (0.4); 
	\node [right] at (4.8, -0.5) {$5$};
		\node [right] at (10.5, -0.5) {$2$};	
		\node [right] at (14.5, -0.5) {$3$};
		\node [right] at (18.5, -0.5) {$4$};
	\node [right] at (14.5, 2) {$1$}; 
		\node [right] at (10.5, 1) {$14$}; 
		\node [right] at (12.7, 1.5) {$9$}; 
		\node [right] at (12.5, -0.5) {$10$}; 
		\node [right] at (16.5, -0.5) {$11$}; 
		
\draw (0.8,-0.5) circle (0.4); 
\draw (0.8,-0.5) circle (0.5); 
\draw (9,-0.5) circle (0.4); 
\draw (9,-0.5) circle (0.5); 
\draw (11,1) circle (0.4); 
\draw (11,1) circle (0.5); 
\draw (13,1.5) circle (0.4); 
\draw (13.1,-0.5) circle (0.4); 
\draw (17.05,-0.6) circle (0.4);
\end{tikzpicture}
\caption {\small{Labeling A applied to $\SS$, $r=1$, $s=x'=2$, $|Y|=3$. The circled and double circled numbers are labeled in Phases II and III, respectively; the rest are labeled in Phase I.}}
\label{big A}
\end{figure}

For the proof, we may assume $\deg(u) = \deg(v) + 1 = 4$ and that $u$ is incident to three pendant edges.
We let $x^{\prime}, r,$ and $s$ be the floors of half the lengths of the middle path, the right odd path with a saved edge $\RO_1$, and the right odd path without a saved edge $\RO_2$, respectively. Thus, we may decompose $m$, the total number of edges, as $m = 2(x^{\prime}+r+s) + 5 + i_{X}$ where $i_{X}$ will be $1$ if $|X|$ is odd, otherwise $i_{X}$ is $0$.

The smallest label on a pendant edge incident to $u$ is at least $r+x^{\prime}+s+1$. The two larger labels on pendant edges incident to $u$ sum to at least $m-i_{X}$. Let $e_{1}$ be the middle path edge incident to $u$. Then $\phi(u) \geq m + x^{\prime} + r + s + f(e_{1}) + 1 - i_{X}$.

Now observe that two of the labels incident to $v$ will always be $m$ and $r+1$. The remaining label may be expressed as $f(e_{1}) + 1 - 2i_{X}$. Thus, $\phi(v) = m + r + f(e_{1}) + 2 - i_{X}$. Now we subtract $\phi(v)$ from our lower bound on $\phi(u)$ to obtain
$\phi(u) - \phi(v) \geq x^{\prime} + s + i_{X} - 1$. Note $s \geq 1$. If $x^{\prime}+s \geq 2$, we are done. If $x^{\prime}+s = 1$, then $x^{\prime} = 0$ but $i_{X} = 1$, so $s+i_{X} = 2$, and the inequality still holds. Thus, in any case, $\phi(u) - \phi(v) \geq 1$.

\medskip

\noindent
{\bf Case 2.} Assume $v$ is incident to at  least one even path, that is $|\RE| \geq 1$. We shall modify Labeling A. The goal is to keep the labels small for the edges incident to $v$ from $\RE$ while maintaining $\phi(u) > \phi(v) > \phi(w)$ for $w$ with $\deg(w) \leq 2$.  To this end, we begin with some special notations.

Denote $a=|\RO|$, $b=|\RE|$, $c = |\LO|$, $d=|\LE|$, the pendant edges incident to $u$ by $Y=\{q_1, q_2, \dots, q_y\}$, and the even paths on the right side by  $ 
\RE_1, \RE_2, \dots, \RE_{b}, $ 
where $|\RE_1| \leq |\RE_2| \dots \leq |\RE_b|$. For each $i$, denote $\RE_i$:  $v=w^i_0, w^i_1, w^i_2,  \dots, w^i_{|\RE_i|}$ with edges  $e^i_{j}=w^i_{j-1}w^i_{j}$.  
By our assumptions, every right odd path has length at least 3, and $b \geq 1$. 
Let $b_2$ be the number of $\RE$ paths of length 2. That is, if $b_2 \geq 1$, then  $|\RE_1| = |\RE_2| = \dots = |\RE_{b_2}|=2$.  Define $\alpha = \max\{0, b-1-(c+d)\}$ and $\beta=\min\{\alpha, b_2\}$.  
If $\beta > 0$, then $\alpha \geq \beta >0$, and there are $\beta$ paths of length 2 on the right side, $\RE_1, \RE_2, \ldots, \RE_{\beta}$. Moreover, as $c+d + y = a+b+1$, $\alpha = b-1-c-d \geq \beta$ and $a \geq 0$, it holds that $y \geq \alpha +2 \geq \beta + 2$.

We now define Labeling B. Besides the right paths introduced in the above, other notations are the same as the ones used in Labeling A, with the following exception:

\begin{center}
\fbox{\begin{minipage}{38 em}\small 
When $x$ is even, $a=c=d=\beta=0$, and $b=2$, the edges of $X$ are ordered reversely. That is,  $$
X = (v=w^X_0, w^X_1, w^X_2, \dots, w^X_x = u), \eqno(*)
$$
with edges $e^X_i=w_{i-1}w_i$, $1 \leq i \leq x$. 
With this reversed order, $X_1$ and $X_2$ are defined the same as in Labeling A.\end{minipage}}
\end{center}

\noindent 
Similar to Labeling A, Labeling B is defined by  three parallel and alternating phases. 

\newpage
\medskip
\noindent
\fbox{{\bf Labeling B}}

\medskip

\noindent
{\bf Phase I}
$$
\small
\begin{array}{lll}
(1)  \ \ e^1_1, q_1, e^2_1, q_2, e^3_1, \dots,  q_{\beta-1},  e^{\beta}_1 \ &\mbox{(if $\beta > 0$)}\\
(2) \ \ \mbox{even edges from the fourth edge of each $\RE_{\beta+1}$, $\ldots$, $\RE_{b-1}$} &\mbox{(if $\beta < b-1$)} \\
(3) \ \ \mbox{$\REE_{\alpha+1} \to  \REE_{\alpha+2} \to \ldots \to  \REE_{b-1}$} &\mbox{(if $\alpha < b-1$)} \\
(4) \ \ \ROO \ \ \mbox{(that is, from the first edge in $\RO_1$ if exists)} \\
(5) \ \ \LOO^* \  \ \mbox{(skipping the last edge on {\bf every} path of  $\LO$)} \\
(6) \ \ X_1 \\ 
(7) \ \ \mbox{even edges on $\RE_b$,  starting with $e^b_ 2$} \\ 
(8) \ \  \LOE \\ 
(9) \ \  e^{(\beta+1)}_1, e^{(\beta+2)}_1, \ldots, e^{\alpha}_1 \ \ \mbox{(first edges of $\RE_{\beta+1}$, $\ldots$, $\RE_{\alpha}$)} &\mbox{(if $\beta < b-1$)}\\ 
(10) \ \  \mbox{the remaining edges in $Y$} 
\end{array} 
$$
{\bf Phase II}
$$
\small
\begin{array}{ll} 
(1)' \ \  e^{\beta}_2,  e^{\beta-1}_2, \ldots, e^{1}_2 \  \mbox{(note, the order is reserved from (1))} &\mbox{(if $\beta > 0$)} \\
(2)' \ \  \mbox{odd edges from the third edge of $\RE_{\beta+1}$, $\RE_{\beta+ 1}$, $\ldots$, $\RE_{\alpha}$}  &\mbox{(if $\beta < b-1$)}\\
(3)' \ \  \mbox{$\ROE_{\alpha+1} \to  \ROE_{\alpha+2} \to \ldots \to  \ROE_{b-1} $} &\mbox{(if $\alpha < b-1$)}\\ 
(4)' \ \  \RO(\E) \\
(5)' \ \ \LO(\E)  \\
(6)' \ \   X_2 \\ 
(7)'  \ \  \mbox{odd edges on $\RE_b$, starting with  $e^{b}_1$} \\
(8)' \ \ \LEE\\ 
(9)' \ \  e^{(\beta+1)}_2, e^{(\beta+2)}_2, \ldots, e^{\alpha}_2  \  \mbox{(second edge of $\RE_{\beta+1}$, $\ldots$, $\RE_{\alpha}$)} &\mbox{(if $\beta < b-1$)}\\
\end{array} 
$$  
\noindent
{\bf Phase III}  
$$
\small 
\begin{array}{llll}
 (11) \ &\mbox{the last edge on every $\LO_1$, $\LO_2$, $\ldots$, $\LO_c$} \\ 
 (12) \  &\mbox{$Z$ \ \ \ (final edges, where $1 \leq |Z| \leq 2$)}. \hspace{2.5in}
\end{array}
$$
Labeling B also has the parallel and alternating ordering properties as Labeling A, with two exceptions. The first exception is when $\beta >0$, then the orderings in (1) and (1)' are reversed. In this case, the degree-2 vertices in $\RE_1$, $\RE_2$, $\ldots$, $\RE_{\beta}$, are $w^i_1$, $1\leq i \leq \beta$, respectively. Assume $f(e^{\beta}_2)=h$. Then $\phi(w^i_1)=(2i-1)+(h+\beta - i)=h + \beta +i -1$, $1 \leq i \leq \beta$. Hence $\phi(w^i_1) \neq \phi(w^j_1)$, $1 \leq i < j \leq \beta$.

The second exception occurs when $\alpha > \beta$, in which, by steps (2), (2)', and (9)', the first three edges  of each $\RE_{\beta+1}$, $\RE_{\beta+2}, \ldots$, $\RE_{\alpha}$ are not  labeled with the alternating property. By the labeling scheme, we have 
$\phi(w_i^{\beta +1}) <  \phi(w_i^{\beta +2}) < \ldots <  \phi(w_i^{\alpha})$ for $i=1,2$. Thus, 
it  suffices to verify that  $\phi(w_1^{\alpha}) < \phi(w_2^{\beta+1})$.   
Assume $f(e_i^{\beta+1}) = t_i$, $1 \leq i \leq 3$.  Then  
$\phi(w_2^{\beta+1})=t_2+t_3$ and 
$\phi(w_1^{\alpha})= t_1+t_2 + 2\alpha - 2\beta -2$. Recall $y \geq \alpha +2$. By steps (8), (9), and $(2)'$, we have $t_3 \geq t_1 + \alpha - \beta + y \geq t_1 + 2\alpha - \beta + 2$.  This implies,  $\phi(w_2^{\beta+1}) =t_2 + t_3 \geq t_1+t_2 + 2\alpha -\beta +2 > t_1+t_2 +2\alpha-2\beta -2=\phi(w_1^{\alpha})$.   

We now conclude that all vertices in $V_1 \cup V_2$ have distinct sums, and $\phi(s) < \phi(s')$ if $s \in V_1$ and $s' \in V_2$.  It remains to show that 
$\phi(u) > \phi(v) > \phi(w)$ holds for any $w \in V_2$. Observe the following:  

\medskip
 
\noindent
The edges incident to $v$ are labeled in the following steps:

(1) \  $\beta$ edges, if $\beta >0$.  


(4) \  $a$ edges, if $a \geq 1$. 

(9) \  $\alpha - \beta$ edges, if $\alpha > \beta$. 

$(3)'$ \  $b - \alpha-1$ edges, if $\alpha < b-1$.

$(6)'$ \ one edge, $e_{1}$ on $X$, if it is  $(*)$: $x$ is even,  $a=\beta=0$, and $b=2$.

$(7)'$ \ one edge, $e^b_{1}$.

(12) \  one edge, $e_x$, labeled with $m$, except for ($*$): $x$ is even, $a=\beta=0$ and $b=2$. 

\medskip

\noindent
The edges incident to $u$ are labeled in the following steps:

(1) \ $\beta-1$ edges, if $\beta >0$.  

(10) \ $y - \beta+1$ edges. 

$(6)'$ \ one edge $e_1$ on $X$, if $x$ is even, and $a \neq 0$ or $b \geq 3$.  

\hspace{0.6in} (note, when $\beta > 0$, then $\alpha > 0$, implying $b \geq 2$). 

$(8)'$ \ $d$ edges, if $d \geq 1$. 

(11) \ $c$ edges, if $c \geq 1$.

(12) \ one edge on $X$, \ 
$
\left\{ 
\begin{array}{lll} 
e_1 &\mbox{if $x$ is odd, labeled with $m-1$}; \\ 
e_x &\mbox{if ($*$) holds: $x$ is even, $a = \beta=0$ and $b=2$, labeled with $m$}.
\end{array}
\right. 
$

\medskip
\noindent
We now prove $\phi(u) > \phi(v)$. By the above  observation, it is enough to consider that $a=c=d=0$, since the $a$ edges incident to $v$ are labeled early on in step (4), while the $y$ edges, the $c$, and the $d$ edges incident to $u$ are labeled later at steps (10), $(8)'$ and (11), respectively. Thus, assume $u$ is adjacent to leaves only and $\alpha=b-1=y-2 \geq 1$.   

Denote the edges incident to $v$ by $t_1, t_2, \ldots, t_{b+1}$ where $f(t_i) < f(t_{i+1})$. Similarly, denote the edges incident to $u$ by $k_1, k_2, \ldots, k_{b+2}$ where $f(k_i) < f(k_{i+1})$. Note, $t_{b+1}$ and $k_{b+2}$ are edges on $X$. Observe:
\begin{itemize}\small
    \item By the end of step (10), $t_1, t_2, \ldots, t_{b-1}$ and $k_1, k_2, \ldots, k_{b+1}$ are labeled, 
    where $f(t_i) < f(k_i)$ for $1 \leq i \leq b-1$. If $\beta > 0$, because at least one label is assigned in step (7), we get $f(t_{\beta})=2 \beta - 1 \leq f(k_{\beta}) - 2$. Thus,  
   
   $
    \sum\limits_{i=1}^{b-1} [ f(k_i) - f(t_i)]
    \geq 
    \left\{
    \begin{array}{llll}
    b &\mbox{if $\beta >0$} \\
    b-1 &\mbox{otherwise.}  
\end{array} 
\right.  
    $
    
Moreover, $f(k_b)$ and $f(k_{b+1})$ are the last assigned in step (10). By the end of step (10), exactly $y + \lfloor \frac{m-y-2}{2} \rfloor= \lfloor \frac{m+y-2}{2} \rfloor$ edges are labeled. As $m+y \equiv x$ (mod 2), we have 

    $
    f(k_{b}) + f(k_{b+1}) =  
    \left\{
    \begin{array}{llll}
    m+y-3 &\geq m &\mbox{if $x$ is even;}\\
m+y-4 &\geq m-1 &\mbox{if $x$ is odd.}  
\end{array} 
    \right.
    $
  \item   
$
f(k_{b+2})= 
\left\{\begin{array}{lll}
m-1 &\mbox{if $x$ is odd (labeled in step (12)),}\\
m &\mbox{if $x$ is even, $a = \beta=0$ and $b=2$ (labeled in step (12)),}\\
f(t_{b}) -1 &\mbox{otherwise (labeled in steps $(6)'$).} 
\end{array}
\right. 
\hspace{1in}
$
\item 
$
f(t_{b+1}) =  
\left\{
\begin{array}{llll}
   m - 2 - \frac{|\RE_b|}{2} &\mbox{if $x$ is even, $a = \beta=0$ and $b=2$,}  \\
   m &\mbox{otherwise.}   
\end{array}  
\right.
$
\end{itemize} 
We  sum up the above to show that $\phi(u) > \phi(v)$. 
If $x$ is odd, 
then $f(t_b) \leq m-2$, so $\phi(u) - \phi(v) \geq (b-1) + (m-1) + (m-1) - (m-2) - m = b-1 \geq 1$. If $x$ is even and $\beta > 0$, then $\phi(u) - \phi(v) \geq b + m + f(t_b) -1 - f(t_b) - m = b-1 \geq 1$. If $x$ is even and $\beta=0$. Recall $b \geq 2$,  
so $\phi(u) - \phi(v) \geq b-1 + m + m - f(t_b) - (m -2 - \frac{|\RE_b|}{2}) = b + 1 + m -f(t_b) + \frac{|\RE_b|}{2} > 0$. 

Next we prove $\phi(v) > \max \{\phi(w):  w \in V_2 \}$. It can be  easily seen true when $x=1$. Assume $x \geq 2$. Except the ($*$) case,  
the maximum value of $\phi(w)$ for  $w \in V_2$ occurs at $w^X_{x-1}$ on $X$, in which $\phi(w)= m + f(e^X_{x-1})$. By our labeling scheme, $e_{x-1}$ is labeled at step $(6)'$, before the first edge of $\RE_b$ which is labeled in step $(7)'$. Hence $\phi(v) > \phi(w)$.   

It remains to consider the ($*$) case when $x$ is even, $a = \beta=0$ and $b=2$, for which the maximum $\phi(w)$ of $w\in V_2$ occurs at $w=w^X_{x-1}$, the vertex on $X$ that is adjacent to $u$. Note, $\phi(w^X_{x-1})=m + f(e^X_{x-1})$. On the other hand, $\phi(v) = f(e^X_1)+f(e^1_1) + f(e_1^2)$. By ($*$), $f(e^X_{x-1}) < f(e^X_1)$ and $f(e^1_1) + f(e_1^2) \geq (\frac{|\RE_1|}{2} + \frac{|\RE_b|}{2}+1) + (m-\frac{|\RE_1|}{2} -1) \geq  m+2$,  as $|\RE_1| \geq 4$.  
Hence, $\phi(v) > \max\{\phi(w): w \in V_2\}$, completing the proof of \Href{2 types} (b).  \hfill$\blacksquare$

See \href{10} for examples.  



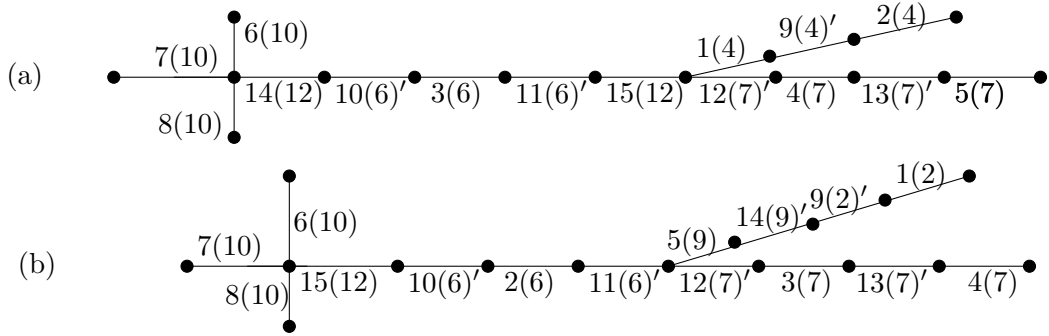
\begin{figure}[h]
\small
\centering
\begin{tikzpicture}[scale = 0.8]
\draw (-2,0) -- (13.5,0); 
\draw (0,0) -- (0,1); 
\draw (0,0) -- (0,-1);
\draw (0,0) -- (-1,0);
\draw (7.5,0) -- (12,1);
\draw[fill=black] (8.9, 0.35) circle (0.1);   \draw[fill=black] (10.3, 0.63) circle (0.1);
\draw[fill=black] (12, 1) circle (0.1);
\draw[fill=black] (9, 0) circle (0.1);    
 \draw[fill=black] (11.8, 0) circle (0.1);
\draw[fill=black] (13.4, 0) circle (0.1);
\draw[fill=black] (10.3, 0) circle (0.1);
\foreach \i in {0,...,5}{
\draw[fill=black] (\i+\i/2, 0) circle (0.1);
	}      
\draw[fill=black] (0,1) circle (0.1);
\draw[fill=black] (-2, 0) circle (0.1);
\draw[fill=black] (0, -1) circle (0.1);
\node [right] at (0, -0.3) {$14  (12)$};
\node [right] at (1.5, -0.3) {$10  (6)'$};
\node [right] at (3.1, -0.3) {$3  (6)$};	
\node [right] at (4.5, -0.3) {$11 (6)'$};
\node [right] at (6, -0.3) {$15  (12)$};
\node [right] at (0, 0.7) {$6  (10)$};	
\node [right] at (-1.5, 0.3) {$7  (10)$};	
\node [left] at (0, -0.8) {$8  (10)$};
\node [left] at (8.7, 0.45) {$1  (4)$};
\node [left] at (10.2, 0.8) {$9  (4)'$};
\node [left] at (11.7, 1) {$2  (4)$};
\node [left] at (9.1, -0.3) {$12  (7)'$};
\node [left] at (10.2, -0.3) {$4  (7)$};
\node [left] at (11.8, -0.3) {$13  (7)'$};
\node [left] at (13, -0.3) {$5  (7)$};
\node [left] at (13, -0.3) {$5  (7)$};
\node [left] at (-3, 0) {(a)};
\end{tikzpicture}


\begin{tikzpicture}[scale = 0.8]
\node [left] at (-4, 0) {(b)};
\draw (-2,0) -- (12,0); 
\draw (-0.3,0) -- (-0.3,1.5); 
\draw (-0.3,0) -- (-0.3,-1);
\draw (0,0) -- (-1,0);
\draw (6,0) -- (11,1.5);
\draw[fill=black] (7.5, 0) circle (0.1);    \draw[fill=black] (9, 0) circle (0.1);
\draw[fill=black] (10.5, 0) circle (0.1);
\draw[fill=black] (12, 0) circle (0.1);  
\draw[fill=black] (7.1, 0.4) circle (0.1);  \draw[fill=black] (8.4, 0.7) circle (0.1);
\draw[fill=black] (9.6, 1.1) circle (0.1);
\draw[fill=black] (11, 1.5) circle (0.1);
	\foreach \i in {1,...,4}{
		    \draw[fill=black] (\i+\i/2, 0) circle (0.1);
	}
\draw[fill=black] (-0.3,1.5) circle (0.1);
\draw[fill=black] (-2, 0) circle (0.1);
\draw[fill=black] (-0.3, 0) circle (0.1);
\draw[fill=black] (-0.3, -1) circle (0.1);
\node [right] at (-0.3, -0.3) {$15 (12)$};
\node [right] at (1.5, -0.3) {$10 (6)'$};
\node [right] at (3.1, -0.3) {$2   (6)$};	
\node [right] at (4.5, -0.3) {$11  (6)'$};
\node [right] at (6, -0.3) {$12 (7)'$};
\node [right] at (-0.4, 0.7) {$6 (10)$};	
\node [right] at (-2, 0.3) {$7 (10)$};	
\node [left] at (-0.1, -0.5) {$8 (10)$};
\node [left] at (7, 0.4) {$5 (9)$};
\node [left] at (8.5, 0.8) {$14 (9)'$};		
\node [left] at (9.5, 1.1) {$9 (2)'$};
\node [left] at (10.8, 1.5) {$1 (2)$};
\node [left] at (8.9, -0.3) {$3 (7)$};
\node [left] at (10.5, -0.3) {$13 (7)'$};
\node [left] at (12, -0.3) {$4 (7)$};
\end{tikzpicture}
\caption {\small{Labeling B: (a) $a=b=1$,  $c=d=0$, $y=3$, $\alpha=\beta=0$, $x=5$; (b) the special case ($*$), $a=\beta=0$, $b=2$, $x=4$, $c=d=0$,  $y=3$, and $\alpha=1$. Each number on the edge is the label while the number in the parentheses is the step that the label was assigned.}}
\label{10}
\end{figure}


Labelings A and B are modified from \cite{strongly}. 
By  \Href{euclid}, one only needs to consider those  double spiders in  \Href{2 types}, resulting in simplified proofs with less cases considered than in \cite{strongly}.  Moreover, if the order of the path $X$ is not revised for the special case ($*$) in Labeling  B as in \cite{strongly},  
direct calculation shows that $\phi(u)=\phi(v)$ when $\alpha=1$; not treating this subtlety caused an error to occur in \cite{strongly}.  For instance, in \href{10} (b), reversing the middle path $X$ gives $\phi(u)=\phi(v)= 32$. 

A {\it cycle double spider} is obtained by replacing each leg of a double spider $\SS$ by a cycle. We may also view such a graph as being obtained by attaching several cycles (of the same or different lengths) to each of the two ends $u$ and $v$ of a path.  

\begin{corollary}
\label{cycle double spider}
A cycle double spider is strongly antimagic. 
\end{corollary}

\begin{proof}
Consider a cycle double spider obtained from a double spider with center vertices $u$ and $v$ where the number of cycles attaching to $u$ is at least the number of cycles attaching to $v$. 
By \href{add a cycle}, we need only show that given a cycle double spider where $u$ and $v$ each has a single cycle attached, we can construct a strongly antimagic labeling $f$ such that $\phi(u) > \phi(v)$. To start, we delete one or several edges from each cycle to form a double spider $\SS$. Then we apply either \href{double leaves} or Labeling A to get a strongly antimagic labeling for $\SS$ where $\phi(u) > \phi(v)$. Finally we use  \Href{big-one} and \Href{add a path} to form the two cycles with the desired lengths.  

Let $n$ and $m$ denote the lengths of the two cycles attaching to $u$ and $v$, respectively. Assume $m=n=3$. We  delete the edge not incident to $u$ or $v$ of the $C_3$ on each side. The resulting graph is a double spider $\SS$ with $u$ and $v$ each adjacent to only two leaves and the middle path. We apply the labeling used in \href{double leaves}. As the labels of the two leaves on each side are consecutive, by  \Href{big-one}, $G$ is strongly antimagic. Note that the labeling \href{double leaves} gives $\phi(u) > \phi(v)$ when $x$ (the length of the middle path) is even, while $\phi(u) < \phi(v)$ when $x$ is odd.  In order to maintain our goal that $\phi(u) > \phi(v)$, we shall swap $u$ and $v$ if $x$ is odd. This concludes the case of $m = n = 3$. 

Next, we prove the case that $n=3$ and $m \geq 4$ (the case for $m=3$ and $n \geq 4$ can be proved similarly). We again proceed by deleting edges to obtain a double spider $\SS$ where each of $u$ and $v$ is incident only to the middle path and two pendant edges and then label $\SS$ using \href{double leaves}. If we further assume that $x$ is odd, the proof follows quickly; in this case, after swapping $u$ and $v$, by \href{double leaves} the two leaves on the left side have the largest $\phi$-values among all leaves. Hence by \Href{big-one}, the graph obtained by adding an edge connecting these two leaves is strongly antimagic. Next we apply \Href{add a path} to connect the two leaves on the $v$ side to get the desired cycle $C_m$.

Next we assume that $n=3$, $m \geq 4$, and $x$  is even. In the following, we refer to the pendant edges incident to $u$ as $e^{\prime}$ and $e^{\prime\prime}$ and those incident to $v$ as $q$ and $q^{\prime}$. Note that the ordering defined in \href{double leaves} ensures that $f(e^{\prime}) < f(e^{\prime\prime}) < f(q) < f(q^{\prime})$.  We consider separately the cases of $m=4$ and $m\geq 5$. We start with the former case and assume that $m=4$. In this case, after labeling the double spider via \href{double leaves}, we add $3$ to every existing label; the labels $1$, $2$, and $3$ are now unused. Then we add a vertex $w$ and add edges from $w$ to the two leaves adjacent to $v$. The edge linking $w$ to the leaf incident to $q^{\prime}$ receives the label $3$, and the edge linking $w$ to the leaf incident to $q$ receives the label $2$. Then we add one more edge between the two leaves adjacent to $u$ and assign this edge the label $1$. We claim the resulting labeling is strongly antimagic. 

First, among vertices whose degrees remained unchanged by the addition of the three edges, the ordering with respect to $\phi$-values is preserved and also each $\phi$-value is increased by at least $6$. On the other hand, for the vertices that were formerly leaves, their $\phi$-values increased by at most $6$. Further, as $f(e^{\prime}) < f(e^{\prime\prime}) < f(q) < f(q^{\prime})$, and the labels  $1$, $2$, and $3$ were redistributed in accordance with this ordering, the $\phi$-value ordering among the four former leaves has been preserved. Thus, the ordering with respect to $\phi$-values has been preserved for all vertices that were in the initial double spider. It is now sufficient that we show $\phi(w)$ is in fact smaller than any $\phi$-value belonging to a former leaf. Note that $\phi(w) = 5$. Since every existing label increased by $3$, the $\phi$-value of each former leaf increased by at least $4$. But each such $\phi$-value was originally at least $2$ since the condition $x \geq 2$ ensures that $f(e^\prime) \geq 2$, so $\phi(w)$ will be the minimal $\phi$-value among all vertices.

We finish by considering $n=3$, $m \geq 5$, and $x$ is even.  The two leaves incident to $q$ and $q^{\prime}$ have the two largest $\phi$-values among the four leaves, so we may apply \Href{big-one} to give a pendant edge to each of the vertices. As the two leaves incident to $e^{\prime}$ and $e^{\prime\prime}$  now maximize $\phi$-value among all leaves, we may again apply  \Href{big-one} to add a new edge between these two leaves. Now we may finish by applying \Href{add a path} to add a path of length $m-4$ between the two remaining leaves. As each extension of the initial double spider preserved the property of being strongly antimagic, the final graph is strongly antimagic.

Assume $m, n \geq 4$. We select a neighbor of $v$ that also belongs to the $C_{m}$ and remove the edge that is incident to the neighbor but not to $v$.  If $n \equiv m$ (mod 2) or $n \equiv 0$ (mod 2), we will select a neighbor of $u$ that also belongs to the $C_{n}$ and remove the edge that is incident to the neighbor but not to $u$. If $n$ is odd and $m$ even, we remove the edge between the two vertices on the $C_{n}$ whose distance from $u$ is $\frac{n-1}{2}$. The two vertices become leaves.  Further, if $n \equiv 3$ (mod 4), we remove these two leaves and their pendant edges.
In this way, we obtain a double spider $\SS$. The legs incident to $v$ have lengths $m-2$  and $1$. In the case of $n$ odd and $m$ even, the two legs incident to $u$ both have length $\frac{n-1-2i_{n}}{2}$ where $i_{n}$ is $1$ if $n \equiv 3$ (mod 4) and $0$ otherwise. In the other cases, the legs incident to $u$ have lengths $n-2$ and $1$. We remove the pendant edge incident to $v$ to obtain a single spider
$\S$. The segment of the $C_{m}$ still incident to $v$ is the right leg, and the two legs incident to $u$ are left legs.

Now we apply Labeling A to $\S$ so that $v$ has the maximal $\phi$-value among all vertices of degree 2. Note that pendant edges on left even paths and pendant edges incident to $u$ are labeled after the right pendant edge. Thus, if $n$ is odd and $m$ is even or $n$ is even, the left pendant edges will have labels exceeding the label on the right pendant edge.
On the other hand, if $n$ is odd and $m$ is odd, the right pendant edge will be labeled before any left edges have been labeled, which again ensures that the two pendant edges on the left have labels exceeding that of the right pendant edge. In any case, the two pendant edges on the left have larger labels than the pendant edge on the right. Now we apply \Href{big-one} once to attach a single pendant edge to $v$, which receives the label $1$. At this stage, we have a double spider $\SS$ where the two pendant edges on the left have the largest labels among the four pendant edges. To finish, we first handle the case where $m$ is even and $n \equiv 3$ (mod 4); here, we apply \Href{big-one} to attach two pendant edges, one to each leaf on the left. The two right pendant edges now have the largest labels among all pendant edges, so we again apply \Href{big-one}, this time creating an edge between the two right leaves. Now the two left leaves are the only leaves in the graph, so we may again apply \Href{big-one} to create an edge connecting the two left leaves. For all other cases, we apply \Href{big-one} twice in succession; we connect the left leaves by an edge and then, as the two right leaves are the only vertices of degree $1$ in the resulting graph, connect the two right leaves by an edge.

All the above labelings have $\phi(u) > \phi(v)$. By \href{add a cycle}, adding a cycle to $u$ remains strongly antimagic labeling. We can do this alternatively  between $u$ and $v$, to get the desired cycle double spider. 
\end{proof}


\section{Conclusion and Open Questions}


Once we know how to strongly antimagic label a given graph $G$, \Href{big-one} allows us to bootstrap this knowledge into the ability to create strongly antimagic labelings for many other graphs, namely those graphs that may be obtained from $G$ by adding edges in a way that respects the vertex ordering induced by the strongly antimagic labeling on $G$. For example, the results of this paper established that spiders and double spiders are strongly antimagic, which easily yielded that cycle spiders and cycle double spiders are strongly antimagic. As a further example, the results of \cite{regular} established that regular graphs are antimagic and thus strongly antimagic, so regular graphs may also be extended to obtain many more strongly antimagic graphs. It is therefore an enticing challenge to establish that certain families of graphs are strongly antimagic as such results may in fact verify the strongly antimagic property for many graphs beyond those contained in the family.

We also observe that results on strongly antimagic labeling may shed light on other topics in antimagic graph labeling. For example, recall the notion of {\it k}-shifted antimagic labeling, which was discussed in Section 1. An ongoing line of research 
\cite{Li2, Li3} seeks to determine which values of {\it k} allow for a given graph to be {\it k}-shifted antimagic. As previously noted, showing a graph is strongly antimagic also proves the graph is {\it k}-shifted antimagic for all nonnegative integer {\it k}. Thus, for example, the results in this paper show double spiders are {\it k}-shifted antimagic for any choice of nonnegative integer {\it k}. 

 
 

It is known that all trees with at most  one vertex of degree 2 are antimagic \cite{partition, trees}. 
It remains an open question whether a strongly antimagic version of this result holds. 
\begin{question} 
\label{q1}
Is every tree with at most one vertex of degree 2 strongly antimagic?
\end{question}

The partition method used in \cite{partition, trees} yields the following  partial solution of \Href{q1}:

\begin{corollary}
A tree with at most one vertex of even degree 
is strongly antimagic. \end{corollary}

A {\it caterpillar} is a tree in which all the vertices are within distance 1 of a central path. It is known that caterpillars are antimagic \cite{caterpillar1, caterpillars, caterpillar all}. The following question remains open: 

\begin{question}
\label{caterpillar}
Is every caterpillar strongly antimagic? 
\end{question}

\noindent
A caterpillar is called {\it regular}  if all non-leaf vertices have the same degree. By  \Href{single} and \Href{path}, one can show that 
every regular caterpillar is  strongly antimagic.  
Moreover, according to  \Href{big-one}, investigating    \Href{caterpillar} bounds to studying the caterpillars $G$ where $V_i(G) \neq \emptyset$ for all $1 \leq i \leq \Delta(G)$. 

In \Href{level-wise regular trees}, we have shown that some level-wise regular trees are strongly antimagic. It is interesting to investigate the following question which is weaker than \Href{q1}:  

\begin{question}
Is every level-wise regular tree strongly antimagic? 
\end{question}


Every strongly antimagic graph is antimagic.  On the other hand, there exist disconnected graphs that are antimagic but not strongly antimagic (cf. \cite{Parker}).  It is natural to ask whether the converse holds for connected graphs.

\begin{question} 
Is every connected antimagic graph strongly antimagic? 
\end{question}





\end{document}